\documentclass[11pt,reqno]{amsart}

\usepackage{slashed}
\usepackage{graphicx}
\usepackage{subfigure}
\usepackage{amsfonts}
\usepackage{amssymb}

\usepackage{latexsym}
\usepackage{amscd}
\usepackage{mathrsfs}
\usepackage{verbatim}
\usepackage{mathrsfs}
\usepackage{enumitem} 
\usepackage{braket}
\usepackage[margin=1.3in]{geometry}
\usepackage{comment}
\usepackage{bm}

\usepackage{imakeidx}
\makeindex[title=Index of Notation]

\usepackage{hyperref}
\hypersetup{
colorlinks=true,
urlcolor=blue,
linkcolor=red,
citecolor=green
}

\usepackage{caption}
\usepackage{array}

\usepackage{amssymb}
\usepackage{pdfpages}
\usepackage{verbatim}

\usepackage{appendix}
\usepackage{xspace}

\newtheorem{theorem}{Theorem}[section]
\newtheorem{lemma}[theorem]{Lemma}
\newtheorem{proposition}[theorem]{Proposition}
\newtheorem{corollary}[theorem]{Corollary}
\theoremstyle{definition}
\newtheorem{definition}[theorem]{Definition}

\theoremstyle{remark}
\newtheorem{remark}[theorem]{Remark}

\makeatletter
\@namedef{subjclassname@2020}{\textup{2020} Mathematics Subject Classification}
\makeatother

\subjclass[2020]{53E10}

\numberwithin{equation}{section}

\begin{document}
\title{  Ancient mean curvature flows from minimal hypersurfaces }

\author{Yongheng Han}

\address{School of Mathematical Science, University of Science and Technology of China, No. 96 Jinzhai Road,  Hefei City,  Anhui Province 230026, China.}

\email{hyh2804@mail.ustc.edu.cn}


\date{\today}
\begin{abstract}
	For $n\geq 2$, we construct $I$-dimensional family of embedded ancient solutions to mean curvature flow emerging from an unstable minimal hypersurface $\Sigma$ with finite total curvature in $\mathbb{R}^{n+1}$, where $I$ is the Morse index of the Jacobi operator on $\Sigma$.
\end{abstract}

\keywords{mean curvature flow, ancient solution, minimal hypersurface}

\maketitle

\section{Introduction}\label{sec1}

A one-parameter family $M_t$ of hypersurfaces in $\mathbb{R}^{n+1}$ flows by mean curvature if
\begin{equation}
	(\partial_t x)^\perp =\textbf{H},
\end{equation}
where $\textbf{H}=-H\textbf{n}$ is the mean curvature vector, $\textbf{n}$ is the outward unit normal, $v^\perp$ is the normal part of a vector $v$,  and the mean curvature $H$ is given by $H=\text{div}(\textbf{n})$.
An ancient solution to mean curvature flow is a solution defined on a time interval of the form  $(-\infty, C]$ with $-\infty<C\leq +\infty$. Ancient solutions are important for  understanding mean curvature flow near singularities. 

Self-shrinkers, translators, and self-expanders are examples of self-similar solutions; in particular, self-shrinkers and translators provide important classes of ancient solutions.
Nguyen \cite{Ngu09,Ngu10,Ngu14} constructs noncomapct self-shrinkers with any genus (see also \cite{Mol11,KKM18,CS23}).
D\'avila-DePino-Nguyen \cite{DDN17} construct self-translating surfaces in $\mathbb{R}^3$ with three parabolic ends and high genus. Hoffman-Ilmanen-Martin-White \cite{HIMW19} classify complete translating graphs in $\mathbb{R}^3$ and produce $(n-1)$-parameter families of translating graphs in $\mathbb{R}^{n+1}$. Sun-Wang \cite{SW23} construct translators with prescribed ends, while Ilmanen \cite{Ilm98} establishes existence results for self-expanding hypersurfaces with prescribed tangent cones at infinity.

Many non-soliton ancient solutions have also been constructed. Bourni-Langford-Mramor \cite{BLM21} construct closed nonconvex nonsoliton ancient mean curvature flows and Mramor-Payne \cite{MP21} obtain ancient and eternal solutions arising from minimal surfaces. Choi-Mantoulidis \cite{CM22} produce ancient gradient flows for elliptic functionals. Bourni-Langford-Tinaglia \cite{BLT22} show existence of ancient solutions emanating from polytopes. Chodosh-Choi-Mantoulidi-Schulze \cite{CCCS24} construct ancient rescaled mean curvature flows asymptotic to a non-compact self-shrinker as  $t \to -\infty$. 

In this paper, we construct an $I$-parameter family of ancient mean curvature flows emerging from a prescribed unstable minimal hypersurface in $\mathbb{R}^{n+1}$ for $n\geq 2$.  Our approach is motivated by the work of Chodosh-Choi-Mantoulidi-Schulze \cite{CCCS24} and may be regarded as a generalization of the construction of Mramor-Payne.

\begin{theorem}\label{thm1}
	Suppose that $\Sigma^n\subset \mathbb{R}^{n+1}$ is an unstable minimal hypersurface with finite total curvature. Let $I$ be the Morse index of $\Sigma$, then there exists $I$-dimensional family of ancient mean curvature flows asymptotic to $\Sigma$ as $t\to -\infty$.
\end{theorem}
\begin{remark}
	Recently \cite{CHL25}, Choi-Huang-Lee give a new proof Theorem \ref{thm1}. Moreover, they get better asymptotic behavior of the ancient solutions. In \cite{Mar24}, Mramor partially classifies the space of eternal mean convex flows in $\mathbb{R}^3$ of finite total curvature type.
\end{remark}

The paper is organized as follows. We first analyze the linearized graphical mean curvature flow equation over an unstable minimal hypersurface. The key difficulty to analysis the equation is that the spectrum of the linearized operator is not discrete. We use the heat kernel estimate to get a weighted $C^0$ estimate. Then, we can easily obtain a weighted  $C^{2,\alpha}$ estimate. We use the weighted Schauder estimates for the linearized parabolic equation to construct mean curvature flows by the contraction principle. The ancient mean curvature flows in Theorem  \ref{thm1} converge to the minimal hypersurface exponentially as $t\to -\infty$.

\noindent
\textbf{Acknowledgment:}
The author is grateful to his advisor, Bing Wang, for his guidance and support. The author is supported by the Project of Stable Support for Youth Team in Basic Research Field,  Chinese Academy of Sciences (YSBR-001). The authors would like to thank the anonymous reviewers for their insightful comments and constructive suggestions, which have significantly improved the quality of this paper.

\section{Proof of Theorem \ref{thm1}}
\subsection{Spectral theory in $L^2$ space} Let $\Sigma\subset \mathbb{R}^{n+1}$ be a smooth minimal hypersurface. We consider the following operator on $\Sigma$:
\begin{equation}
	\begin{split}
		Lu:=\Delta_\Sigma u+|A_\Sigma|^2u.
	\end{split}
\end{equation}
This is the stability operator for area functional in the sense that
\begin{equation}
	\begin{split}
		\frac{d^2}{ds^2}\bigg|_{s=0}\text{Area}(\text{graph}_\Sigma(su) )=-\int_\Sigma u(Lu)d\mathcal{H}^n.
	\end{split}
\end{equation}
We define a $L^2$ product for measurable functions $u,v:\Sigma\to \mathbb{R}:$
\begin{equation}
	\begin{split}
		\langle u,v\rangle:=\int_\Sigma uvd\mathcal{H}^n.
	\end{split}
\end{equation}
This induces a metric$\|\cdot\|$ and a Hilbert space
\begin{equation}
	\begin{split}
		L^2(\Sigma):=\{u:\Sigma\to \mathbb{R}:\|u\|<+\infty\}.
	\end{split}
\end{equation}
Likewise, we define the higher order Sobolev spaces
\begin{equation}
	\begin{split}
		H^k(\Sigma):=\{u:\Sigma\to \mathbb{R}:\|u\|+\|\nabla_\Sigma u\|+\cdots \|\nabla_\Sigma ^k u\|<+\infty\}.
	\end{split}
\end{equation}
It is with respect to these measures spaces that $L$ is self-adjoint, i.e.,
\begin{equation}
	\begin{split}
		\langle Lu,v\rangle=\langle u,Lv\rangle\;\forall u,v\in H^2(\Sigma).
	\end{split}
\end{equation}
We say $\Sigma$ has finite total curvature if $\int_{ \Sigma}|A|^nd\mu<+\infty$.
\begin{definition}
	The Morse index of a compact minimal hypersurface $\Sigma$ is the number of negative eigenvalues of the stability operator $L$ (counting the multiplicity) acting on the space of smooth functions which vanish on the boundary.
	
	Suppose $\Sigma$ is a noncompact minimal hypersurface. $\Sigma_i\subset \Sigma$ is compact with smooth boundary which satisfies $\Sigma_i\subset\Sigma_{i+1}$ and $\Sigma=\cup_i\Sigma_i$. If $\lim_{i\to \infty}index(\Sigma_i)<+\infty$, we say $\Sigma$ has finite index and the index of $\Sigma$ is $\lim_{i\to \infty}index(\Sigma_i)$. $\Sigma$ is called stable if $index(\Sigma)=0$, otherwise unstable. 
\end{definition}
\begin{lemma}\label{lm2}
	Suppose that $\Sigma^n\in \mathbb{R}^{n+1}$ is an unstable minimal hypersurface with finite total curvature, there exists $I\in \mathbb{N}^+$ and real numbers $\lambda_1< \lambda_2\leq \cdots\leq \lambda_{I}<0$, such that $\sigma(-L)=\{\lambda_1, \lambda_2, \cdots,\lambda_{I}\}\cup [0,+\infty)$. Also there is a  corresponding 
	$L^2$-orthonormal basis $\phi_1,\cdots,\phi_I$ such that $L\phi_i+\lambda_i\phi_i=0$. Moreover, $\Sigma$ has finite end.
\end{lemma}
\begin{proof}
	Li-Yau \cite{LY83} show that the number of nonnegative eigenvalues of $-L$ is bounded by $c(n)\int_{ \Sigma}|A|^nd\mu$. So, index$(\Sigma)<+\infty$. Since the spectrum of $-\Delta$ is $[0,+\infty)$(cf.\cite{Sch83}) and  $|A|^2$ is a bounded operator, the essential spectrum of $-\Delta-|A|^2$ and $-\Delta$ are the same. 
	
	In \cite{Li17}, Li proves that
	\begin{equation}
		\frac{2}{n(n+1)}(\#ends +b_1(\Sigma)-1)\leq index (\Sigma)+nullity(\Sigma)\leq C\int_\Sigma |A|^n,
	\end{equation}
	where $nullity(\Sigma)$ is the dimension of the space of $L^2$ solutions of the Jacobi operator,
	and $b_1(M)$ is the first Betti number of the compactification of $\Sigma$.
\end{proof}

We next show that eigenfunctions corresponding to negative eigenvalues must decay exponentially.
\begin{lemma}\label{lm2-3}
	The same assumption  as in Lemma \ref{lm2}. Since $\phi_1$ does not change sign, we assume that $\phi_1>0$. For any $\varepsilon\in (0,\frac12)$, there is a constant $C=C(\Sigma,\varepsilon)>0$ such that  
	\begin{equation}
		0<\phi_1(x)\leq Ce^{-\sqrt{-(1-\varepsilon)\lambda_1} |x|}.
	\end{equation}
	Moreover, 
	\begin{equation}\label{req10}
		|\phi_i|\leq Ce^{-\sqrt{-(1-\varepsilon)\lambda_i} |x|}.
	\end{equation}
\end{lemma}
\begin{proof}
	For any positive integer $j$, denote
	\begin{equation}
		L_j= (\Delta +|A|^2)|_{B_j(p)\cap \Sigma}
	\end{equation}
	where $p\in \Sigma$ is a fixed point.
	Let $\lambda_{j,i}$ and $\phi_{j,i}$ be the eigenvalues and eigenfunctions  of $L_j$ with Dirichlet boundary conditions. 
	For $i\leq I$, we can assume 
	\begin{equation}\label{eq2.11}
		\lim\limits_{j\to \infty} \lambda_{j,i}\to\lambda_i \text{ and } \phi_{j,i}\overset{C_{loc}^\infty}{\longrightarrow} \phi_i.
	\end{equation}
	Since $|A|\to 0$ as $|x|\to\infty$, for $x$ sufficiently large we have
	\begin{equation}\label{eq2.12}
		\Delta \phi_{j,1}(x)=(-\lambda_{j,1}-|A|^2)\phi_{j,1}(x)\geq -(1-\tfrac{1}{3}\varepsilon )\lambda_{j,1}\phi_{j,1}(x).
	\end{equation}
We assume that $0 \notin M$ and let $|x|$ denote the extrinsic distance. By direct calculation, we have
	\begin{equation}\label{eq2.13}
		\nabla|x|=\frac{1}{|x|}\nabla(\tfrac12 |x|^2)=\frac{x^T}{|x|}
	\end{equation}
	and
	\begin{equation}\label{eq2.14}
		\begin{split}
			\Delta |x|=\frac{1}{|x|}\Delta (\tfrac12 |x|^2)-\frac{|\nabla(\frac12|x|^2)|^2}{|x|^3}=\frac{n}{|x|}-\frac{|x^T|^2}{|x|^3}
		\end{split}
	\end{equation}
	where $x^T$ denotes the orthogonal projection of $x$ onto the tangent space $T_x\Sigma$.
Since the tangent cone of $\Sigma$ at infinity consists of planes (counted with multiplicity), it follows that
	\begin{equation}\label{eq2.15}
		\lim\limits_{|x|\to  \infty}|\nabla |x||^2=\lim\limits_{|x|\to  \infty}\frac{|x^T|^2}{|x|^2}=1,\quad  \lim\limits_{|x|\to  \infty}\Delta  |x|=0.
	\end{equation}
	Define the auxiliary function $f$ by
	\begin{equation}
		f=e^{-\sqrt{-(1-\varepsilon)\lambda_1} |x|}.
	\end{equation}
	Then, we have 
	\begin{equation}
		\begin{split}
			\Delta f=f (-(1-\varepsilon)\lambda_1 |\nabla |x||^2- \sqrt{-(1-\varepsilon)\lambda_1}\Delta |x|).
		\end{split}
	\end{equation}
	Combining \eqref{eq2.11}, \eqref{eq2.12} and \eqref{eq2.15}, we can choose $R_0>0,j_0\in \mathbb{N}$ such that for $|x|\geq R_0,j>j_0$:
	\begin{equation}\label{eq2.18}
		\Delta f\leq -(1-\frac{2}{3}\varepsilon)\lambda_1f,\quad  \Delta \phi_{j,1}\geq  -(1-\frac{1}{2}\varepsilon )\lambda_1\phi_{j,1} .
	\end{equation}
	Suppose $\frac{\phi_{j,1}}{f}$ achieves its positive maximum at some $x_0\in x\in B_{j}\setminus B_{R_0}$. Then
	\begin{equation}
		\nabla \frac{\phi_{j,1}}{f}=0,\quad \Delta \frac{\phi_{j,1}}{f}\leq 0.
	\end{equation}
	However, we compute
	\begin{equation}
		\begin{split}
			\Delta\left(\frac{\phi_{j,1}}{f}\right)=\frac{f\Delta\phi_{j,1}-\phi_{j,1}\Delta f}{f^2}-2\frac{1}{f}\langle \nabla f,\nabla \frac{\phi_{j,1}}{f}\rangle =\frac{f\Delta\phi_{j,1}-\phi_{j,1}\Delta f}{f^2}.
		\end{split}
	\end{equation}
By \eqref{eq2.18}, it follows that
	\begin{equation}
		f\Delta\phi_{j,1}-\phi_{j,1}\Delta f \geq -\frac{1}{6}\lambda_1\varepsilon\phi_{j,1}f>0,
	\end{equation}
	which contradicts $\Delta (\phi_{j,1}/f)\leq 0$.
	Hence, the maximum of $\phi_{j,1}/f$ must occur inside $B_{R_0}$:
	\begin{equation}
		\sup\limits_{x\in B_{j}}\frac{\phi_{j,1}}{f} \leq  \sup\limits_{x\in B_{R_0}}\frac{\phi_{j,1}}{f}
	\end{equation}
for any $j\geq j_0$.	Passing to the limit as $j\to \infty$ gives the exponential decay estimate:
	\begin{equation}
		\phi_1\leq  Cf\leq Ce^{-\sqrt{-(1-\varepsilon)\lambda_1} |x|} \;\mathrm{for}\;x\in \Sigma\setminus B_{R_0}.
	\end{equation}
	A similar argument applies to $\phi_{j,i}^+=\mathrm{max}\{\phi_{j,i},0\}$ and $\phi_{j,i}^-=\mathrm{max}\{-\phi_{j,i},0\}$ which proves \eqref{req10}.
\end{proof}

\subsection{Weighted H\"older space}

Let $(\Sigma,g)$ be a noncompact Riemannian manifold. Define $\rho(x)=(1+d(x,p)^2)^{1/2}$ where $p\in \Sigma$ is a fixed point and $\delta(g)$ is the injectivity radius. We introduce the weight H\"older spaces.

\begin{definition}\label{w-holder}(weight H\"older spaces)
	For $\beta \in\mathbb{R}$ and $k$ a nonnegative integer define $C^k_{\beta}$ to be the space of continuous function $f$ on $\Sigma\times\mathbb{R}$ with $k$ continuous derivatives. Define  the norm $\|\cdot\|_{C^k_{\beta}}$ on $C^k_{\beta}(\Sigma\times \mathbb{R})$ by
	\begin{equation}
		\begin{split}
			\|f\|_{C^k_{\beta}}=\sum_{2i+|l|\leq k}\sup\limits_{(x,t)\in \Sigma\times\mathbb{R}} \rho ^{2i+|l|+\beta}|D_t^iD^l f|(x,t).
		\end{split}
	\end{equation}
	where $l$ is a multi-index. For $T$ a tensor field on $\Sigma\times\mathbb{R}$ and $\alpha\in (0,1) ,\beta \in \mathbb{R}$, define
	\begin{equation*}
		\begin{split}
			[T]_{\alpha,\beta}(x,t)=\sup\limits_{(x,t)\neq (y,s) ,d(x,y)\leq \delta(g)} \left[|(\rho(x)+\rho(y)+|s-t|^{1/2})^{\beta} \frac{|T(x,s)-T(y,t)|}{d(x,y)^{\alpha}+|s-t|^{\alpha/2}} \right]
		\end{split}
	\end{equation*}
	and 
	\begin{equation*}
		\begin{split}
			[T]_{\alpha,\beta}=\sup\limits_{(x,t)\in \Sigma \times \mathbb{R}}[T]_{\alpha,\beta}(x,t).
		\end{split}
	\end{equation*}
	For $\beta\in \mathbb{R}$, $k$ an nonnegative integer, and $\alpha\in (0,1)$, define the weighted H\"older space $C^{k,\alpha}_{\beta}$ to be the set of $f\in C^k_\beta(\Sigma\times \mathbb{R})$ for which the norm
	\begin{equation}
		\|f\|_{C^{k,\alpha}_{\beta}}=\|f\|_{C^k_{\beta}}+\sup\limits_{2i+|l|=k}[D^i_tD^l f]_{\alpha,\beta+k+\alpha}
	\end{equation}
	is finite.
	
	For any fixed $t$, we define
	\begin{equation}
		\begin{split}
			\|f(\cdot,t)\|_{C^{k,\alpha}_{\beta}}&=\sum_{2i+|l|\leq k}\sup\limits_{x\in \Sigma} \rho ^{2i+|l|+\beta}|D_t^iD^l f|(x,t)\\
			&+\sum_{2i+|l|= k}\sup\limits_{x\in \Sigma}[D^i_tD^l f]_{\alpha,\beta+k+\alpha}(x,t).
		\end{split}
	\end{equation}
\end{definition}
We obtain the following weighted Schauder estimate on minimal hypersurface $\Sigma$. Noting that there is $C>0$ such that
\begin{equation}
	\begin{split}
		C^{-1} (1+|x|)\leq \rho(x)\leq C (1+|x|).
	\end{split}
\end{equation} 
In other word, $\rho(x)$ is comparable to $1+|x|$, and we write $\rho (x)\sim 1+|x|$.

\begin{lemma}\label{WS}
	Suppose $u,h$ satisfy the following inhomogeneous linear parabolic equation 
	\begin{equation}
		(\partial_t-\Delta-|A|^2)u=h
	\end{equation}
	and $\sup\|u(\cdot, \sigma)\|_{C^{0}_\beta},\sup\|h(\cdot, \sigma)\|_{C^{0,\alpha}_{\beta+2}}<\infty$, for some $\beta\in\mathbb{R},\alpha\in (0,1)$. Then,
	\begin{equation}
		\|u(\cdot, \tau)\|_{C^{2,\alpha}_\beta} \leq C \sup _{\sigma \leq \tau}\left(\|u(\cdot, \sigma)\|_{C^{0}_\beta}+\|h(\cdot, \sigma)\|_{C^{0,\alpha}_{\beta+2}}\right).
	\end{equation}
	
\end{lemma}
\begin{proof}
	In a bounded domain, standard local estimates apply to show that 
	\begin{equation}
		\|u(\cdot, \tau)\|_{C^{2,\alpha}_\beta} \leq C \sup _{\sigma \leq \tau}\left(\|u(\cdot, \sigma)\|_{C^{0}_\beta}+\|h(\cdot, \sigma)\|_{C^{0,\alpha}_{\beta+2}}\right).
	\end{equation}
	
	In what remains, we shall apply a scaling argument. Taking $R$ large enough, we can write $\Sigma\cap B_{R}^c$ as graphs of a fixed plane(cf.\cite{Sch83}). In each end of $\Sigma$, we have 
	\begin{equation}
		\begin{split}
			\partial_tu&=\frac{1}{\sqrt{\mathrm{det}g}}\partial_i(\sqrt{\mathrm{det}g}g^{ij}\partial_j u)+|A|^2u+h\\
			&=g^{ij}\partial_{ij} u+\partial_i(\sqrt{\mathrm{det}g}g^{ij})\partial_j u+|A|^2u+h,
		\end{split}
	\end{equation}
	where $\{\partial_j\}_{j=1}^n$ is an orthogonal basis of $\mathbb{R}^{n}$.
	Since each end of $\Sigma$ is of finite total curvature, Schoen \cite{Sch83} shows that $|\nabla^l \partial_i(\sqrt{\mathrm{det}g}g^{ij})|\leq C_l\frac{1}{1+|x|^l}$
	when $|x|\geq R_0$. 
	
	For	any  $r_0\geq R_0$, define the rescaled function
	\begin{equation}
		w\doteqdot r^\beta_0u(r_0x,t+t_0+\tau r_0^2).
	\end{equation}
	Then 
	\begin{equation}
		\partial_{ij} w=r_0^{\beta+2} \partial_{ij} u,\quad \partial_{i} w=r_0^{\beta+1} \partial_{i} u,
	\end{equation}
	$w$ solves the equation
	\begin{equation}
		\partial_{\tau}w=g^{ij}\partial_{ij} w+r_0\partial_i(\sqrt{\mathrm{det}g}g^{ij})\partial_j w +r_0^2|A|^2w+r_0^{\beta+2}h.
	\end{equation}
	Noting that $|\nabla^jA|\leq \frac{C}{1+|x|^{1+j}}$ and $ |\partial_i(\sqrt{\mathrm{det}g}g^{ij})|\leq \frac{C}{1+|x|}$. For $ (x,\tau)\in (B_{4}(0)\setminus B_1(0))\times[-2,0]$,
	interior estimates for parabolic equations then imply that for some
	constant $C = C(n, R_0, \alpha,\beta)$
	\begin{equation}
		\begin{split}
			&\|w\|_{C^{2,\alpha}( (B_{3}(0)\setminus B_2(0))\times[-1,0])}\\
			&\leq C\|w\|_{C^{0}( (B_{4}(0)\setminus B_1(0))\times[-2,0])}\\
			&+Cr_0^{\beta+2}\|h\|_{C^{0,\alpha}( (B_{4}(0)\setminus B_1(0))\times[-2,0])}.
		\end{split}
	\end{equation}
	Which, in turn, implies that
	\begin{equation}
		\begin{split}
			&\|u\|_{C^{2,\alpha}_\beta( (B_{3r_0}(0)\setminus B_{2r_0}(0))\times[t_0-r_0^2,t_0])}\\
			&\leq C\|u\|_{C^{0}_\beta( (B_{4r_0}(0)\setminus B_{r_0}(0))\times[t_0-2r_0^2,t_0])}\\
			&+C\|h\|_{C^{0,\alpha}_{\beta+2}( (B_{4r_0}(0)\setminus B_{r_0}(0))\times[t_0-2r_0^2,t_0])}.
		\end{split}
	\end{equation}
	By the arbitrariness of $r_0,t_0$,  
	\begin{equation}
		\|u(\cdot, \tau)\|_{C^{2,\alpha}_\beta} \leq C \sup _{\sigma \leq \tau}\left(\|u(\cdot, \sigma)\|_{C^{0}_\beta}+\|h(\cdot, \sigma)\|_{C^{0,\alpha}_{\beta+2}}\right).
	\end{equation}
\end{proof}

\subsection{Heat kernel estimate}
In this subsection, we establish global bounds for the heat kernel 
$G$ of the Schrödinger operator $\Delta+|A|^2$. Although Zhang \cite{Zha01} obtained both upper and lower bounds for $G$, his estimates cannot be applied directly to our setting. We therefore refine and improve these bounds. We begin by recalling a well-known result.
\begin{proposition}\cite{Dav}
	Suppose that $\Sigma$ is a minimal hypersurface with finite total curvature, then there exists a heat kernel $G(x,y,t)\in C^{\infty}(\Sigma\times \Sigma\times\mathbb{R}^+)$ such that $(e^{L t}f)(x)=\int_{\Sigma}G(x,y,t)f(y)d\mu,\;\forall f\in L^2(\Sigma)$ and
	\begin{equation*}
		\begin{split}
			&	(1)\quad G(x,y,t)=G(y,x,t),\\
			&(2)\quad \lim\limits_{t\to 0^+}G(y,x,t)=\delta_x(y),\\
			&(3)\quad (L-\frac{\partial}{\partial t})G=0,\\
			&(4)\quad G(x,y,t)=\int_{\Sigma}G(x,z,t-s)G(z,y,s)dz.
		\end{split}
	\end{equation*}
\end{proposition}
Next, we establish both lower and upper bounds for the heat kernel  $G$.
\begin{proposition}\label{prop2.8}
	Suppose that $\Sigma^n\subset\mathbb{R}^{n+1}$ is a minimal hypersurface with finite total curvature. Denote $G(x,y,t)$ to be the heat kernel corresponding to the equation $(L-\frac{\partial}{\partial t})G=0$. For any $\delta>0$, there exists $C>0$ and $c_i>0,i=1,2$ depending only $\Sigma,\delta$ such that for $t\geq 1$
	\begin{equation}
		\begin{split}
			\left |G(x,y,t)-\sum_{\lambda_{i}<0 }e^{-\lambda_{i}t}\phi_{i}(x)\phi_{i}(y)\right|
			&\leq Ce^{\delta t}\left( e^{-c_1d^2(x,y)/t}+  e^{-c_2d(x)-c_2d(y)} \right)
		\end{split}
	\end{equation}
	where $d(x)=d(x,p)$.
\end{proposition}

\begin{proof}
	Let $G$ be the heat kernel of the schr\"odinger  operator $L=\Delta+|A|^2$. Since $\Delta+|A|^2\leq \Delta+\sup |A|^2$, we get
	\begin{equation}\label{eq2.22}
		G(x,y,t)\leq e^{\sup |A|^2t}p(x,y,t)\leq C\frac{e^{C_0t}}{ t^{n/2}}e^{-\frac{d^2(x,y)}{8t}},
	\end{equation}
	where $C_0=\sup |A|^2$ and $p(x,y,t)$ is the heat kernel of $\Delta$ on $\Sigma$.
	For a positive integer $j$, let $G_j$ be the Dirichlet heat kernel of $L_j\equiv (\Delta +|A|^2)|_{B_j(0)\cap \Sigma}$ on $B_j(0)\cap \Sigma$. Obviously, $G_j$ is non-decreasing as $j$ increases and $\lim\limits_{j\to \infty}G_j(x,y,t)=G(x,y,t)$ for all $x,y\in \Sigma$ and $t>0$. Since $B_j(0)\cap \Sigma$ is compact, we have
	\begin{equation}
		G_j(x,y,t)=\sum_{\lambda_{j,i}<0}e^{-\lambda_{j,i}t}\phi_{j,i}(x)\phi_{j,i}(y)+\sum_{\lambda_{j,i}\geq 0}e^{-\lambda_{j,i}t}\phi_{j,i}(x)\phi_{j,i}(y),
	\end{equation}
	where $\lambda_{j,i}$ and $\phi_{j,i}$ are the eigenvalues and normalized eigenfunctions  of $L_j$. Since $G_j(x,x,t)\leq G(x,x,t)\leq \frac{e^{C_0t}}{(4\pi t)^{n/2}}$, we have 
	\begin{equation}
		\sum_{\lambda_{j,i}\geq 0}e^{-\lambda_{j,i}}\phi_{j,i}^2(x)\leq G(x,x,1)\leq C_1.
	\end{equation}
	This shows that $\sum_{\lambda_{j,i}\geq 0}e^{-\lambda_{j,i}t}\phi_{j,i}^2(x)\leq  C_1$ for $t\geq 1$. Hence Cauchy-Schwarz's inequality implies
	\begin{equation}\label{eq2.27}
		\bigg|\sum_{\lambda_{j,i}\geq 0}e^{-\lambda_{j,i}t}\phi_{j,i}(x)\phi_{j,i}(y)\bigg|\leq C_1.
	\end{equation}
	Letting $j\to \infty$, we obtain
	\begin{equation}
		\bigg|G(x,y,t)-\sum_{\lambda_{i}< 0}e^{-\lambda_{i}t}\phi_{i}(x)\phi_{i}(y)\bigg|\leq C_1,
	\end{equation}
	for $t\geq 1$, where $\lambda_i$ is the eigenvalue of $L$. Here, we use that $\lambda_{j,i}\to \lambda_i$ and $\phi_{j,i}\to \phi_i$ in any compact domain on $M$.
	
	Denote $N_j\equiv$ the number of negative eigenvalues of $L_j\equiv
	(\Delta +|A|^2)|_{B_j(0)\cap \Sigma}$. There exist $j_0,N\in\mathbb{N}^+$ such that $N_j=N$ for $j>j_0$. By Lemma \ref{lm2-3} and $\rho(p,x)\sim 1+|x|$, we have
	\begin{equation}
		\left|\sum_{\lambda_{i}<0}e^{-\lambda_{i}t}\phi_{i}(x)\phi_{i}(y)\right|\leq C_1e^{-\lambda_{1}t-C_2d(x)-C_2d(y)},
	\end{equation}
	where $d(x)=d(p,x)$ and $d(y)=d(p,y)$ for a fixed $p\in \Sigma$ and $d$ denotes the intrinsic distance on $M$. For any fixed $x,y$, we know that when $j$ is sufficiently large
	\begin{equation}
		\left|\sum_{\lambda_{j,i}<0}e^{-\lambda_{j,i}t}\phi_{j,i}(x)\phi_{j,i}(y)\right|\leq C_3e^{-\lambda_{j,1}t-C_2d(x)-C_2d(y)}.
	\end{equation}
	Hence, 
	\begin{equation}\label{eq2.29}
		\begin{split}
			\left |\sum_{\lambda_{j,i}\geq 0}e^{-\lambda_{j,i}t}\phi_{j,i}(x)\phi_{j,i}(y)\right|&\leq G_j(x,y,t)+\left|\sum_{\lambda_{j,i}<0}e^{-\lambda_{j,i}t}\phi_{j,i}(x)\phi_{j,i}(y)\right|\\
			&\leq C\frac{e^{C_0t}}{ t^{n/2}}e^{-\frac{d^2(x,y)}{8t}}+ C_3e^{-\lambda_{1}t-C_2d(x)-C_2d(y)}.\\
		\end{split}
	\end{equation}
	Combining \eqref{eq2.29} with \eqref{eq2.27},
	we obtain
	\begin{equation}
		\begin{split}
			&\left |\sum_{\lambda_{j,i}\geq 0}e^{-\lambda_{j,i}t}\phi_{j,i}(x)\phi_{j,i}(y)\right|\\
			\leq &C_1^{1-\delta/C_4}\left(C\frac{e^{C_0t}}{ t^{n/2}}e^{-\frac{d^2(x,y)}{8t}}+ C_3e^{-\lambda_{1}t-C_2d(x)-C_2d(y)}\right)^{\delta/C_4}\\
			\leq &Ce^{\delta t}\left( e^{-c_1d^2(x,y)/t}+  e^{-c_2d(x)-c_2d(y)} \right),
		\end{split}
	\end{equation}
	for any  $0<\delta<1,t\geq 1$. Here $c_i,C_j$ is independent of $x,y$ and $C_4=\text{max}\{-\lambda_1,C_0\}$.
	Taking limit, we have 
	\begin{equation}
		\begin{split}
			\left |G(x,y,t)-\sum_{\lambda_{i}<0 }e^{-\lambda_{i}t}\phi_{i}(x)\phi_{i}(y)\right|
			&\leq Ce^{\delta t}\left( e^{-c_1d^2(x,y)/t}+  e^{-c_2d(x)-c_2d(y)} \right).
		\end{split}
	\end{equation}
\end{proof}

\subsection{$C^{2,\alpha}$ estimates} We fix $\delta_0\in (0,-\lambda_I), \alpha\in (0,1)$ throughout the section.  It will be convenient to consider the operator
\begin{equation}
	\begin{split}
		\iota_-:\textbf{a}=(a_1,\dots,a_I)\in\mathbb{R}^I\mapsto \sum_{j=1}^{I}a_je^{-\lambda_i\tau}\phi_j.
	\end{split}
\end{equation}
For $\mu\leq 0$, we define the spectral projector $\Pi_{<\mu}:L^2(\Sigma)\to L^2(\Sigma)$ given by:
\begin{equation}
	\begin{split}
		\Pi_{<\mu}:f\mapsto \sum_{j:\lambda_j<\mu}\langle f,\phi_j\rangle \phi_j.
	\end{split}
\end{equation}
and 
\begin{equation}
	\begin{split}
		f_j=\langle f,\phi_j\rangle \phi_j,1\leq j\leq I\text{ and }\tilde{f}=f-\Pi_{<0}f.
	\end{split}
\end{equation}
We revisit the inhomogeneous linear PDE
\begin{equation}\label{eq3.7}
	\begin{split}
		(\frac{\partial}{\partial t}-L)u=h\text{ on }\Sigma\times \mathbb{R}_-,
	\end{split}
\end{equation}
where $\mathbb{R}_-=(-\infty,0]$.
\begin{lemma}
	Fix $\delta>0$, $0<\delta^\prime<\mathrm{min}\{\delta,-\lambda_I\}$. Suppose that  $u,h$ satisfy \eqref{eq3.7} and
	\begin{equation}
		\int_{-\infty}^0\bigg|e^{-\delta\tau}\|h(\cdot,\tau)\|_{L^2}\bigg|^2d\tau<\infty.
	\end{equation}
	There is a unique solution $u$ of \eqref{eq3.7} such that 
	\begin{equation}
		\Pi_{<0}(u(\cdot,0))=0,
	\end{equation}
	\begin{equation}\label{eq38}
		\int_{-\infty}^0\bigg|e^{-\delta^\prime\tau }\left(\|u(\cdot,\tau)\|_{L^2}+\|u_\tau(\cdot,\tau)\|_{L^2}\right)\bigg|^2d\tau<\infty.
	\end{equation}
	It is given by 
	\begin{equation}
		\begin{split}
			u_j(x,t)&:=-\int_{t}^{0}e^{\lambda_j(s-t)}h_j(x,s)ds\\\
			\tilde{u}(x,t)&:=\int_{-\infty}^{t}\int_\Sigma G^{\geq 0}(x,y,t-s)\tilde{h}(y,s)d\mu ds\\
			&=\int_{-\infty}^{t}\int_\Sigma G(x,y,t-s)\tilde{h}(y,s)d\mu ds,
		\end{split}
	\end{equation}
	since $\langle\tilde{h},\phi_j\rangle=0$ for $1\leq j\leq I$.
	Moreover, for every $t\in \mathbb{R}_-$,
	\begin{equation}\label{eq2.40}
		e^{-\delta' t}\|u(\cdot,t)\|_{L^2}\leq C\bigg[\int_{-\infty}^0\bigg|e^{-\delta\tau}\|h(\cdot,\tau)\|_{L^2}\bigg|^2d\tau\bigg]^{1/2},
	\end{equation}
	where $C=C(\delta,\delta^\prime,\lambda_1,\cdots,\lambda_I)$ and $G^{\geq 0}:= G(x,y,t)-\sum_{i=1 }^Ne^{-\lambda_{i}t}\phi_{i}(x)\phi_{i}(y)$ is the ``heat kernel" of the nonnegative part of $L$.
\end{lemma}
\begin{proof}
	
	\textit{Step 1. existence and uniqueness.}\\
	
	We first show that $u(x,t)=\sum u_j(x,t)+\tilde{u}(x,t)$ is a solution of \eqref{eq3.7}. By direct calculation 
	\begin{equation}
		\frac{d}{dt}u_j(x,t)=\lambda_j\int_{t}^{0}e^{\lambda_j(s-t)}h_j(x,s)ds+h_j(x,t),
	\end{equation}
	and 
	\begin{equation}
		\begin{split}
			Lu_j(x,t)&=-L\int_{t}^{0}e^{\lambda_j(s-t)}h_j(x,s)ds=\lambda_j\int_{t}^{0}e^{\lambda_j(s-t)}h_j(x,s)ds.
		\end{split}
	\end{equation}
	So, we have 
	\begin{equation}
		\frac{d}{dt}\sum_{j=1}^I u_j(x,t)=L\sum_{j=1}^I u_j(x,t)+\sum_{j=1}^I h_j(x,t).
	\end{equation}
	Similarly,
	\begin{equation}
		\begin{split}
			\left(\frac{d}{dt}-L\right)\tilde{u}(x,t)&=\int_{-\infty}^{t}\int_\Sigma \left(\frac{d}{dt}-L\right)G(x,y,t-s)\tilde{h}(y,s)d\mu ds\\
			&+\lim_{s\to t}\int_\Sigma G(x,y,t-s)\tilde{h}(y,s)d\mu\\
			&=\tilde{h}(x,t).
		\end{split}
	\end{equation}
	For uniqueness, if $u,v$ are solutions of \eqref{eq3.7}, then $(\partial_t-L)(u-v)=0$. Hence, 
	\begin{equation}
		\frac{d}{dt}\|u_j-v_j\|^2=-2\lambda_j\|u_j-v_j\|^2.
	\end{equation}
	Since  $u_j(x,0)=v_j(x,0)$, we know that $u_j\equiv v_j$ and $\tilde{u}-\tilde{v}=u-v$.
	\begin{equation}
		\begin{split}
			\frac{d}{dt}\|\tilde{u}-\tilde{v}\|^2=2\langle \partial_t(\tilde{u}-\tilde{v}),(\tilde{u}-\tilde{v}) \rangle=2\langle L(\tilde{u}-\tilde{v}), (\tilde{u}-\tilde{v}) \rangle\leq 0.
		\end{split}
	\end{equation}
	For any $t$, we have
	\begin{equation}
		\begin{split}
			\|\tilde{u}(\cdot,t)-\tilde{v}(\cdot,t)\|^2\leq \int^{t-n}_{t-n-1}\|\tilde{u}(\cdot,\tau)-\tilde{v}(\cdot,\tau)\|^2d\tau.
		\end{split}
	\end{equation}
	Letting $n\to\infty$, we get $ \|\tilde{u}(\cdot,t)-\tilde{v}(\cdot,t)\|^2=0$, i.e., $u\equiv v$.\\
	
	\textit{Step 2.  prove the inequality \eqref{eq2.40}.}\\
	
	By Cauchy-Schwartz's inequality
	\begin{equation}\label{req49}
		\begin{split}
			\|u_j(\cdot,t)\|^2&=\int_\Sigma \bigg(\int_{t}^{0}e^{\lambda_j(s-t)}h_j(x,s)ds\bigg)^2dx\\
			&\leq  \int_{t}^{0}e^{2\lambda_j(s-t)+2\delta s}ds\int_t^0\int_\Sigma e^{-2\delta s}h^2_j(x,s)dxds\\
			&\leq e^{2\mathrm{min}\{-\lambda_I,\delta\} t}\int_{-\infty}^0\bigg|e^{-\delta\tau}\|h(\cdot,\tau)\|_{L^2}\bigg|^2d\tau\\
			&\leq  e^{2\delta' t}\int_{-\infty}^0\bigg|e^{-\delta\tau}\|h(\cdot,\tau)\|_{L^2}\bigg|^2d\tau.\\
		\end{split}
	\end{equation}
	For $\tilde{u}$, we have
	\begin{equation}
		\begin{split}
			\frac{d}{dt}\|\tilde{u}\|^2=2\langle \partial_t\tilde{u},\tilde{u} \rangle=2\langle L\tilde{u}+\tilde{h}, \tilde{u} \rangle\leq 2\langle \tilde{h}, \tilde{u} \rangle\leq 2\|\tilde{h}\|\|\tilde{u}\|. \\
		\end{split}
	\end{equation}
	So, 
	\begin{equation}\label{req51}
		\begin{split}
			&\|\tilde{u}(\cdot,t)\|\leq \int^{t}_{-\infty}\|\tilde{h}(\cdot,\tau)\|d\tau\\
			&\leq \bigg(\int^t_\infty e^{2\delta \tau}\bigg)^\frac{1}{2}\bigg(\int^{t}_{-\infty}e^{-2\delta \tau}\|\tilde{h}(\cdot,\tau)\|^2d\tau\bigg)^\frac{1}{2}\\
			&\leq  Ce^{\delta \tau}\bigg(\int^{t}_{-\infty}e^{-2\delta \tau}\|\tilde{h}(\cdot,\tau)\|^2d\tau\bigg)^\frac{1}{2}.\\
		\end{split}
	\end{equation}
	Combining \eqref{req49} and \eqref{req51}, we get \eqref{eq2.40}.\\
	
	\textit{Step 3.  prove the inequality \eqref{eq38}.}\\
	
	By \eqref{req49} and \eqref{req51}, we have 
	\begin{equation}
		\begin{split}
			\int_{-\infty}^0e^{-2\delta^\prime t }\|u(\cdot,t)\|^2dt\leq  C\int_{-\infty}^0 e^{2(\mathrm{min}\{-\lambda_I,\delta\}-\delta')t}\int^{t}_{-\infty}e^{-2\delta \tau}\|h(\cdot,\tau)\|^2d\tau<\infty.\\
		\end{split}
	\end{equation}
The estimate $\int_{-\infty}^0 e^{-\delta^\prime \tau} \|u_t(\cdot, \tau)\|_{L^2} d\tau < \infty$ is established through a standard energy argument involving integration by parts. We refer the reader to Evans \cite{Eva10} for a comprehensive derivation of such weighted energy estimates.
\end{proof}
\begin{lemma}\label{Lm2.8}
	Fix $\beta>n$, $\mathbf{a}\in \mathbb{R}^I$ and $0<\delta_0<-\lambda_I$. Suppose $u,h$ satisfy \eqref{eq3.7}. If
	\begin{equation}
		\sup_{\tau\leq 0}e^{-2\delta_0 \tau}	\|h(\cdot,\tau)\|_{C_{\beta+2}^{0,\alpha}}< \infty,
	\end{equation}
	\begin{equation}
		\Pi_{<0}(u(\cdot,0))=\iota_{-}(\mathbf{a})(\cdot,0),
	\end{equation}
	and
	\begin{equation}
		\int_{-\infty}^0\bigg|e^{-\delta_0\tau }(\|u(\cdot,\tau)\|_{L^2}+\|u_\tau(\cdot,\tau)\|_{L^2})\bigg|^2d\tau<\infty.
	\end{equation}
	Then, for every $\tau\leq 0$, 
	\begin{equation}
		e^{-\delta_0\tau}\|u(\cdot,\tau)-\iota_{-}(\mathbf{a})\|_{C^{0}_{\beta}}\leq C\sup\limits_{\tau\leq 0}e^{-2\delta _0\sigma}\|h(\cdot,\tau)\|_{C^{0}_{\beta}},
	\end{equation}
	with $C=C(\delta_0,\Sigma,\beta)$.
\end{lemma}

\begin{proof}
	The constant $C$ in the proof may be different from line to line but is independent of $x$. Recall that $\rho(x)=(1+d(x,p)^2)^{1/2}\sim 1+|x|$.	Without loss of generality, we assume that $\mathbf{a}=\mathbf{0}$.
	
	\noindent \textit{ Step 1. We prove the lemma for $d(x,p)\geq 1$}.
	
	\textit{Step 1.1 We bound $\|\tilde{u}(\cdot,t)\|_{C^0_\beta}$ by $\|\tilde{h}\|_{C^0_\beta}$.}\\
	
	For $ t-s\geq 1$, let $\delta=\delta_0/2$ in Proposition \ref{prop2.8}, we have
	\begin{equation}\label{req52}
		\begin{split}
			&\bigg|\int_{ \Sigma}G^{\geq 0}(x,y,t-s)\tilde{h}(y,s)d\mu\bigg|\\
			&\leq Ce^{\delta_0(t-s)/2}\int_{ \Sigma}\left(e^{-\frac{c_1d^2(x,y)}{(t-s)}}+e^{-c_2d(x)-c_2d(y)}\right)
			|\tilde{h}(y,s)|d\mu\\
			&\leq  C\|\tilde{h}(\cdot,s)\|_{C^0_{\beta}}e^{\delta_0(t-s)/2}\int_{ \Sigma}e^{-\frac{c_1d^2(x,y)}{(t-s)}}\frac{1}{\rho^{\beta}(y)}d\mu\\
			&+ C\|\tilde{h}(\cdot,s)\|_{C^0_{\beta}}e^{\delta_0(t-s)/2}\int_{ \Sigma}e^{-c_2d(x)-c_2d(y)}\frac{1}{\rho^{\beta}(y)}d\mu.\\
		\end{split}
	\end{equation}
	Next, we show 
	\begin{equation}
		\begin{split}
			I:=\int_{ \Sigma}e^{-\frac{c_1d^2(x,y)}{(t-s)}}\frac{1}{\rho^{\beta}(y)}d\mu\leq e^{\delta_0(t-s)/2}\frac{C}{\rho^{\beta}(x)} 
		\end{split}
	\end{equation}
	and
	\begin{equation}
		II:=\int_{ \Sigma}e^{-c_2d(x)-c_2d(y)}\frac{1}{\rho^{\beta}(y)}d\mu\leq \frac{C}{\rho^{\beta}(x)} .
	\end{equation}
	For $I$, we can decompose $\Sigma$ as 
	\begin{equation}
		A_1\cup A_2:=\bigg\{y\in \Sigma:d(p,y)\leq \frac{d(x,p)}{2}\bigg\}\cup \bigg\{y\in \Sigma:d(p,y)>\frac{d(x,p)}{2}\bigg\}.
	\end{equation}
	Then, we obtain
	\begin{equation}\label{eq2.53}
		\begin{split}
			\int_{ \Sigma}e^{-\frac{c_1d^2(x,y)}{(t-s)}}\frac{1}{\rho^{\beta}(y)}d\mu&=\left\{\int_{ A_1}+\int_{A_2}\right\}e^{-\frac{c_1d^2(x,y)}{(t-s)}}\frac{1}{\rho^{\beta}(y)}d\mu:=I_1+I_2.\\
		\end{split}
	\end{equation}
	We next bound the integrals of the two terms on the right. 
	For any $y\in A_1$,
	\begin{equation}
		d(x,y)\geq d(x,p)-d(p,y)\geq \frac{d(x,p)}{2}.
	\end{equation}
	So, we have
	\begin{equation}\label{eq2.48}
		\begin{split}
			\int_{ A_1}e^{-\frac{c_1d^2(x,y)}{(t-s)}}\frac{1}{\rho^{\beta}(y)}d\mu
			\leq e^{-\frac{c_1d^2(p,x)}{4(t-s)}}\int_{A_1}\frac{1}{\rho^{\beta}(y)}d\mu
			\leq e^{-\frac{c_1d^2(p,x)}{4(t-s)}}\int_{A_1} 1 d\mu.
		\end{split}
	\end{equation}
	Since  $\Sigma$  has (intrinsic and extrinsic) at most Euclidean volume growth, we know that 
	\begin{equation}
		\int_{A_1} 1 d\mu= \mathrm{Vol}(A_1)\leq Cd^n(x,p)\leq C\rho^n(x).
	\end{equation}
	So,
	\begin{equation}\label{r58}
		\begin{split}
			I_1  &\leq C\rho^n(x)e^{-\frac{c_1d^2(x,p)}{4(t-s)}}
			\\
			&=C\frac{\rho^{n+\beta}(x)}{d^{n+\beta}(x,p)}\frac{(t-s)^{(n+\beta)/2}}{\rho^{\beta}(x)}\left(\frac{d^2(x,p)}{t-s}\right)^{(n+\beta)/2}e^{-\frac{c_1d^2(x,y)}{4(t-s)}}\\
			&\leq C(t-s)^{(n+\beta)/2}\rho^{-\beta}(x)   
		\end{split}
	\end{equation}
	where we use  $\sup_{\tau\geq 0} \tau ^{(n+\beta)/2}e^{-\frac{c_1}{4}\tau}<\infty$.
	
	For the second term $I_2$,
	\begin{equation}\label{eq2.49}
		\begin{split}
			&\int_{ A_2}e^{-\frac{c_1d^2(x,y)}{(t-s)}}\frac{1}{\rho^{\beta}(y)}d\mu\\
			&\leq  \frac{C}{\rho^{\beta}(x)} \int_{ A_2}e^{-\frac{c_1d^2(x,y)}{(t-s)}}d\mu\\
			&\leq \frac{C}{\rho^{\beta}(x)}\int_{\Sigma}e^{-\frac{c_1d^2(x,y)}{(t-s)}}d\mu\\
			&\leq C(t-s)^{n/2}\rho^{-\beta}(x).
		\end{split}
	\end{equation}
	Combining \eqref{r58} and \eqref{eq2.49}, we have
	\begin{equation}\label{req60}
		I\leq C(t-s)^{(n+\beta)/2}\rho^{-\beta}(x)+C(t-s)^{n/2}\rho^{-\beta}(x)\leq Ce^{\delta_0(t-s)/2}\rho^{-\beta}(x).
	\end{equation}
	For  $II$, by direct calculation, we have
	\begin{equation}\label{req61}
		II=\int_{ \Sigma}e^{-c_2d(x)-c_2d(y)}\frac{1}{\rho^{\beta}(y)}d\mu\leq Ce^{-c_2d(x)}\leq C\rho^{-\beta}(x).
	\end{equation}
	Plugging \eqref{req60} and \eqref{req61} into \eqref{req52}, we get
	\begin{equation}\label{req62}
		\begin{split}
			\bigg|\int_{ \Sigma}G^{\geq 0}(x,y,t-s)\tilde{h}(y,s)d\mu\bigg|
			\leq C\|\tilde{h}(\cdot,s)\|_{C^0_{\beta}}e^{\delta_0(t-s)}\rho^{-\beta}(x).
		\end{split}
	\end{equation}
	For $0< t-s\leq 1$, \eqref{eq2.22} implies
	\begin{equation}
		\begin{split}
			&\bigg|\int_{ \Sigma}G^{\geq0}(x,y,t-s)\tilde{h}(y,s)d\mu\bigg|\\
			&= \bigg|\int_{ \Sigma}G(x,y,t-s)\tilde{h}(y,s)d\mu\bigg|\\
			&\leq \frac{C}{(t-s)^{n/2}}\int_{ \Sigma}e^{-\frac{d^2(x,y)}{8(t-s)}}
			|\tilde{h}(y,s)|d\mu\\
			&\leq  C\|\tilde{h}(\cdot,s)\|_{C^0_{\beta}}\frac{C}{(t-s)^{n/2}}\int_{ \Sigma}e^{-\frac{d^2(x,y)}{8(t-s)}}\frac{1}{\rho^{\beta}(y)}d\mu.\\
		\end{split}
	\end{equation}
	Similar to \eqref{r58} and \eqref{eq2.49}, we have 
	\begin{equation}\label{req64}
		\begin{split}
			&\int_{ \Sigma}{(t-s)^{-n/2}}e^{-\frac{d^2(x,y)}{8(t-s)}}\frac{1}{\rho^{\beta}(y)}d\mu\\
			&={(t-s)^{-n/2}}\left\{\int_{ A_1}+\int_{A_2}\right\}e^{-\frac{d^2(x,y)}{8(t-s)}}\frac{1}{\rho^{\beta}(y)}d\mu\\
			&\leq C\frac{1+(t-s)^{\beta/2}}{\rho^{\beta}(x)}\\
			&\leq  \frac{C}{\rho^{\beta}(x)}.
		\end{split}
	\end{equation}
	Combining \eqref{req62} and \eqref{req64}, we get 
	\begin{equation}
		\begin{split}
			\rho^{\beta}(x)|\tilde{u}(\cdot,t)|&\leq \rho^{\beta}(x) \left\{\int_{t-1}^t+\int_{-\infty}^{t-1}\right\}|G^{\geq 0}*\tilde{h}(\cdot,s)|ds\\
			&\leq  C\int_{t-1}^t\|\tilde{h}(\cdot,s)\|_{C^0_{\beta}}	ds+C\int_{-\infty}^{t-1}\|\tilde{h}(\cdot,s)\|_{C^0_{\beta}}e^{\delta_0(t-s)}ds.
		\end{split}
	\end{equation}
	By direct calculation, we have 
	\begin{equation}
		\begin{split}
			\int_{t-1}^t\|\tilde{h}(\cdot,s)\|_{C^0_{\beta}}ds\leq e^{\delta_0 t}\sup_{\tau\in [t-1,t]}\left(e^{-2\delta_0 \tau}\|\tilde{h}(\cdot,\tau)\|_{C^0_{\beta}}\right),
		\end{split}
	\end{equation}
	and 
	\begin{equation}
		\begin{split}
			&\int_{-\infty}^{t-1}\|\tilde{h}(\cdot,s)\|_{C^0_{\beta}}e^{\delta_0(t-s)}ds\\
			&=\int_{-\infty}^{t-1}\bigg(e^{-2\delta_0s}\|\tilde{h}(\cdot,s)\|_{C^0_{\beta}}\bigg)e^{\delta_0(t+s)}ds\\
			&\leq \sup_{\tau\leq t-1}\left(e^{-2\delta_0 \tau}\|\tilde{h}(\cdot,\tau)\|_{C^0_{\beta}}\right)\int_{-\infty}^{t-1}e^{\delta_0(t+s)}ds\\
			&\leq Ce^{\delta_0t}\sup_{\tau\leq t}\left(e^{-2\delta_0 \tau}\|\tilde{h}(\cdot,\tau)\|_{C^0_{\beta}}\right).
		\end{split}
	\end{equation}
	So,
	\begin{equation}
		\rho^{\beta}(x)|\tilde{u}(\cdot,t)|\leq Ce^{\delta_0t}\sup_{\tau\leq t}\left(e^{-2\delta_0 \tau}\|\tilde{h}(\cdot,\tau)\|_{C^0_{\beta}}\right).
	\end{equation}
	
	\textit{Step 1.2 We show $\|h_j(\cdot,t)\|_{C^0_\beta},\|\tilde{h}(\cdot,t)\|_{C^0_\beta}\leq C\|h(\cdot,t)\|_{C^0_\beta}$.}\\
	On the other hand, 
	\begin{equation}
		\begin{split}
			|h_j(\cdot,t)|=|\langle h,\phi_j\rangle \phi_j|\leq \|h(\cdot,t)\|_{C^0_\beta} \|\phi_j\|_{C^0}\int_{\Sigma}\rho^{-\beta}(x)\cdot |\phi_j|\leq C\|h(\cdot,t)\|_{C^0_\beta}. 
		\end{split}
	\end{equation}
	where we use the exponential decay of $\phi_j$. Hence,
	\begin{equation}\label{req71}
		\begin{split}
			\|\tilde{h}(\cdot,t)\|_{C^0_\beta}	= \|h(\cdot,t)-\sum_{j=1}^{I}h_j\|_{C^0_\beta}\leq C\|h(\cdot,t)\|_{C^0_\beta}
		\end{split}
	\end{equation}
	which implies
	\begin{equation}
		\begin{split}
			\rho^{\beta}(x)|\tilde{u}(\cdot,t)|\leq Ce^{\delta_0t}\sup_{\tau\leq t}\left(e^{-2\delta_0 \tau}\|h(\cdot,\tau)\|_{C^0_{\beta}}\right).
		\end{split}
	\end{equation}
	
	\textit{Step 1.3. We control $\|u_j\|_{C^0_{\beta}}$.}\\
	
	For  $u_j$, 
	\begin{equation}
		\begin{split}
			\rho^{\beta}(x)|u_j(x,t)|&=	\rho^{\beta}(x)\bigg|\int_{t}^0e^{\lambda_j(s-t)}h_j(x,s)ds\bigg|\\
			&\leq C\int_{t}^{0}e^{\lambda_j(s-t)}\|h(\cdot,s)\|_{C^0_\beta}ds\\
			&\leq C\sup_{s\leq 0}\left(e^{-2\delta_0 s}\|h(\cdot,s)\|_{C^{0}_{\beta}}\right)\int_{t}^{0}e^{\lambda_j(s-t)+\delta_0s}ds\\
			&\leq Ce^{\delta_0 t} \sup_{s\leq 0}e^{-2\delta_0 s}\|h(\cdot,s)\|_{C^{0}_{\beta}}
		\end{split}
	\end{equation}
	where we use 
	\begin{equation}
		\begin{split}
			\int_{t}^{0}e^{\lambda_j(s-t)+\delta_0s}ds&=\frac{1}{\lambda_j+\delta_0}e^{-\lambda_jt}e^{(\lambda_j+\delta_0)s}\bigg|^0_t\\
			&=\frac{1}{-(\lambda_j+\delta_0)}e^{-\lambda_jt}\bigg(e^{(\lambda_j+\delta_0)t}-1\bigg)\\
			&\leq Ce^{\delta_0 t}.
		\end{split}
	\end{equation}
	
	\textit{Step 1.4. We get the conclusion.}\\
	Noting that $u=\sum_{j} u_j+\tilde{u}$, we arrive at
	\begin{equation}
		e^{-\delta_0t}\|u(\cdot,t)-\iota_{-}(\mathbf{a})\|_{C^{0}_{\beta}(\Sigma\setminus B_1(p))}\leq C\sup_{\tau\leq 0}e^{-2\delta_0 \tau}\|h(\cdot,\tau)\|_{C^{0}_{\beta}}.
	\end{equation}
	
	\noindent \textit{ Step 2. We prove the lemma for $d(x,p)\leq 1$}.
	
	\textit{Step 2.1 We bound $\|\tilde{u}(\cdot,t)\|_{C^0}$ by $\|\tilde{h}\|_{C^0_\beta}$.}\\
	For $ t-s\geq 1$, let $\delta=\delta_0/2$ in Proposition \ref{prop2.8}, we have
	\begin{equation}\label{req76}
		\begin{split}
			&\bigg|\int_{ \Sigma}G^{\geq 0}(x,y,t-s)\tilde{h}(y,s)d\mu\bigg|\\
			&\leq Ce^{\delta_0(t-s)/2}\int_{ \Sigma}\left(e^{-\frac{c_1d^2(x,y)}{(t-s)}}+e^{-c_2d(x)-c_2d(y)}\right)
			|\tilde{h}(y,s)|d\mu\\
			&\leq  C\|\tilde{h}(\cdot,s)\|_{C^0_\beta}e^{\delta_0(t-s)/2}\int_{ \Sigma}e^{-\frac{c_1d^2(x,y)}{(t-s)}}d\mu\\
			&+ C\|\tilde{h}(\cdot,s)\|_{C^0_\beta}e^{\delta_0(t-s)/2}\int_{ \Sigma}e^{-c_2d(x)-c_2d(y)}d\mu\\
			&\leq C\|\tilde{h}(\cdot,s)\|_{C^0_\beta}e^{\delta_0(t-s)}. 
		\end{split}
	\end{equation}
	For $0\leq t-s\leq 1$, \eqref{eq2.22} implies
	\begin{equation}\label{req77}
		\begin{split}
			&\bigg|\int_{ \Sigma}G^{\geq0}(x,y,t-s)\tilde{h}(y,s)d\mu\bigg|= \bigg|\int_{ \Sigma}G(x,y,t-s)\tilde{h}(y,s)d\mu\bigg|\\
			&\leq \frac{C}{(t-s)^{n/2}}\int_{ \Sigma}e^{-\frac{d^2(x,y)}{8(t-s)}}
			|\tilde{h}(y,s)|d\mu\\
			&\leq  C\|\tilde{h}(\cdot,s)\|_{C^0_\beta}.
		\end{split}
	\end{equation}
	Combining \eqref{req76} and \eqref{req77}, we get 
	\begin{equation}
		\begin{split}
			|\tilde{u}(\cdot,t)|&\leq  \left\{\int_{t-1}^t+\int_{-\infty}^{t-1}\right\}|G^{\geq 0}*\tilde{h}(\cdot,s)|ds\\
			&\leq  C\int_{t-1}^t\|\tilde{h}(\cdot,s)\|_{C^0_{\beta}}	ds+C\int_{-\infty}^{t-1}\|\tilde{h}(\cdot,s)\|_{C^0_{\beta}}e^{\delta_0(t-s)}ds\\
			&\leq Ce^{\delta_0t}\sup_{\tau\leq t}\left(e^{-2\delta_0 \tau}\|\tilde{h}(\cdot,\tau)\|_{C^0_{\beta}}\right).
		\end{split}
	\end{equation}
	
	\textit{Step 2.2. We control $\|u_j\|_{C^0}$.}\\
	For $u_j$, 
	\begin{equation}
		\begin{split}
			|u_j(x,t)|&=\bigg|\int_{t}^0e^{\lambda_j(s-t)}h_j(x,s)ds\bigg|\\
			&\leq C\int_{t}^{0}e^{\lambda_j(s-t)}\|h(\cdot,s)\|_{C^0_\beta}ds\\
			&\leq C\sup_{s\leq 0}\left(e^{-2\delta_0 s}\|h(\cdot,s)\|_{C^{0}_\beta}\right)\int_{t}^{0}e^{\lambda_j(s-t)+\delta_0s}ds\\
			&\leq Ce^{\delta_0 t} \sup_{s\leq 0}e^{-2\delta_0 s}\|h(\cdot,s)\|_{C^{0}_\beta}.
		\end{split}
	\end{equation}
	
	\textit{Step 2.3 We get the conclusion.}\\
	Noting that $\rho^\beta(x)\sim 1$, we have
	\begin{equation}
		\begin{split}
			\rho^\beta(x)|u_j(x,t)|\leq Ce^{\delta_0 t} \sup_{s\leq 0}e^{-2\delta_0 s}\|h(\cdot,s)\|_{C^{0}_\beta},
		\end{split}
	\end{equation}
	and
	\begin{equation}
		\begin{split}
			\rho^\beta(x)|\tilde{u}(\cdot,t)|\leq Ce^{\delta_0t}\sup_{\tau\leq t}\left(e^{-2\delta_0 \tau}\|\tilde{h}(\cdot,\tau)\|_{C^0_{\beta}}\right).
		\end{split}
	\end{equation}
	Finally, we prove 
	\begin{equation}
		e^{-\delta_0t}\|u(\cdot,t)-\iota_{-}(\mathbf{a})\|_{C^{0}_{\beta}}\leq C\sup_{\tau\leq 0}e^{-2\delta_0 \tau}\|h(\cdot,\tau)\|_{C^{0}_{\beta}}.
	\end{equation}
\end{proof}
We have the following global weighted $C^{2,\alpha}$ estimate
\begin{proposition}\label{prop2.9}
	Suppose $u,h$ satisfy \eqref{eq3.7}. Given $0<\delta_0<-\lambda_I$, there exists $C=C(\Sigma,\alpha,\beta,\delta_0)$ such that
	\begin{equation}
		\begin{split}
			e^{-\delta_0 t}\|u(\cdot,t)-\iota_{-}(\mathbf{a})\|_{C^{2,\alpha}_\beta}\leq C\sup\limits_{\tau\leq 0}e^{-2\delta_0 \tau}\|h(\cdot,\tau)\|_{C^{0,\alpha}_{\beta+2}},
		\end{split}
	\end{equation}
	for all $t\in\mathbb{R}_-$$(\mathbb{R}_{-}=(-\infty,0])$.
\end{proposition}
\begin{proof}
	By Lemma \ref{WS} and Lemma \ref{Lm2.8}, we get 
	\begin{equation}
		\begin{split}
			&e^{-\delta_0t}\|u(\cdot,t)-\iota_{-}(\mathbf{a})\|_{C^{2,\alpha}_{\beta}}\\
			&\leq C\sup_{\tau\leq t}\left(e^{-\delta_0\tau}\|u(\cdot,\tau)-\iota_{-}(\mathbf{a})\|_{C^{0}_{\beta}}+e^{-\delta_0\tau}\|h(\cdot,\tau)\|_{C^{0,\alpha}_{2+\beta}}\right)\\
			&\leq C\sup_{\tau\leq t}\left(e^{-2\delta_0 \tau}\|h(\cdot,\tau)\|_{C^{0}_{\beta}}+e^{-\delta_0\tau}\|h(\cdot,\tau)\|_{C^{0,\alpha}_{2+\beta}}\right)\\
			& \leq C\sup_{\sigma\leq 0}e^{-2\delta_0 \tau}\|h(\cdot,\tau)\|_{C^{0,\alpha}_{2+\beta}}.
		\end{split}
	\end{equation}
\end{proof}

\subsection{Nonlinear error term.} 

Define the graph 
\begin{equation}
	\Gamma_t:=\{\mathbf{x}+u(\mathbf{x},t)\nu_{\Sigma}(\mathbf{x}):\mathbf{x}\in \Sigma\}.
\end{equation}
In graphical coordinates over $\Sigma$, the mean curvature flow is:
\begin{equation}\label{eq3.16}
	\begin{split}
		\frac{\partial}{\partial t}u=vH_{\Gamma_t}.
	\end{split}
\end{equation}
We can rewrite \eqref{eq3.16} as 
\begin{equation}\label{eq3.18}
	\begin{split}
		(\frac{\partial}{\partial t}-L)u=E(u)\text{ on }\Sigma\times \mathbb{R}_-,
	\end{split}
\end{equation}
since the mean curvature $ H_{\Gamma_t}$ of $\Gamma_t$ at $\mathbf{x}+u(\mathbf{x},t)\nu_{\Sigma}(\mathbf{x})$ satisfies
\begin{equation}
	v(\mathbf{x,t})H_{\Gamma_t} (\mathbf{x}+u(\mathbf{x},t)\nu_{\Sigma}(\mathbf{x}))=\left(\Delta_{\Sigma}u+|A|^2u\right)(\mathbf{x})+E(u),
\end{equation}
where $v$ is the geometry function
\begin{equation}
	\begin{split}
		v:=(1+|(\text{Id}-uA_\Sigma)^{-1}(\nabla_\Sigma u)|^2)^{\frac{1}{2}}=(\nu\cdot\nu_\Sigma)^{-1},
	\end{split}
\end{equation}
and the nonlinear error term $E(u)$ can be estimated as follows:

\begin{lemma}\label{lmm2.11}\cite{CCCS24}
	For $\beta>0$, there exists $\eta=\eta(\Sigma,\beta)$ such that for $u:\Sigma\to \mathbb{R}$ with $\|u\|_{C^{2,0}_{\beta}}\leq \eta$, the nonlinear error term decomposes as 
	\begin{equation}\label{eq2.54}
		\begin{split}
			E(u)(\mathbf{x})&=u(\mathbf{x})E_1(\mathbf{x},u(\mathbf{x}),\nabla_\Sigma u(\mathbf{x}),\nabla^2_\Sigma u(\mathbf{x}))\\
			&+\nabla_\Sigma u(\mathbf{x})\cdot\mathbf{E}_2(\mathbf{x},u(\mathbf{x}),\nabla_\Sigma u(\mathbf{x}),\nabla^2_\Sigma u(\mathbf{x})).
		\end{split}
	\end{equation}
	where $E_1,E_2$ are smooth functions on the following domains:
	\begin{equation*}
		\begin{split}
			&E_1(\mathbf{x},\cdot,\cdot,\cdot):\mathbb{R}\times T_{\mathbf{x}}\Sigma\times \mathrm{Sym}( T_{\mathbf{x}}\Sigma\otimes  T_{\mathbf{x}}\Sigma)\to \mathbb{R}\\
			& \textbf{E}_2(\mathbf{x},\cdot,\cdot,\cdot):\mathbb{R}\times T_{\mathbf{x}}\Sigma\times \mathrm{Sym}( T_{\mathbf{x}}\Sigma\otimes  T_{\mathbf{x}}\Sigma)\to  T_{\mathbf{x}}\Sigma.
		\end{split}
	\end{equation*}
	Moreover, we can estimate:
	\begin{equation}\label{eq2.55}
		\begin{split}
			r(\mathbf{x})^{2+j-l}|\nabla_{\mathbf{x}}^i\nabla_{z}^j\nabla_{\mathbf{q}}^k\nabla_{\mathbf{A}}^lE_1(\mathbf{x},z,\mathbf{q},\mathbf{A})|\leq C(r(\mathbf{x})^{-1}|z|+|\mathbf{q}|+r(\mathbf{x})|\mathbf{A}|),
		\end{split}
	\end{equation}
	\begin{equation}\label{eq2.56}
		\begin{split}
			r(\mathbf{x})^{1+j-l}|\nabla_{\mathbf{x}}^i\nabla_{z}^j\nabla_{\mathbf{q}}^k\nabla_{\mathbf{A}}^l\mathbf{E}_2(\mathbf{x},z,\mathbf{q},\mathbf{A})|\leq C(r(\mathbf{x})^{-1}|z|+|\mathbf{q}|+r(\mathbf{x})|\mathbf{A}|),
		\end{split}
	\end{equation}
	In the above, $C=C(\Sigma,\beta)$, $r(x)=1+|\mathbf{x}|$ and $i,j,k,l\geq 0$.
\end{lemma}

Moreover, the nonlinear error term $E(u)$ can be estimated as follows:
\begin{corollary}\label{coro2.12}\cite{CCCS24}
	For $\beta>0$, there exists $\eta=\eta(\Sigma,\beta)$ such that for $u:\Sigma\to \mathbb{R}$ with $\|u\|_{C^{2,0}_{\beta}}\leq \eta$:
	\begin{equation}\label{eq2.57}
		\begin{split}
			r|E(u)|\leq C(r^{-1}|u|+|\nabla_\Sigma u|)(r^{-1}|u|+|\nabla_\Sigma u|+r|\nabla_\Sigma^2 u|),
		\end{split}
	\end{equation}
	\begin{equation}\label{eq2.58}
		\begin{split}
			\|E(u)\|_{C^{0,\alpha}_{\beta+2}}\leq C \|u\|_{C^{1,\alpha}_{\beta}}\|u\|_{C^{2,\alpha}_{\beta}},
		\end{split}
	\end{equation}
	and for $\bar{u}:\Sigma\to \mathbb{R}$ also with $\|\bar{u}\|_{C^{2,0}_\beta}\leq \eta$:
	\begin{equation}\label{eq2.59}
		\begin{split}
			r|E(\bar{u})-E(u)|\leq &C(r^{-1}|u|+|\nabla_\Sigma u|+r|\nabla_\Sigma^2 u|+r^{-1}|\bar{u}|+|\nabla_\Sigma \bar{u}|+r|\nabla_\Sigma^2 \bar{u}|)\\
			&\cdot(r^{-1}|\bar{u}-u|+|\nabla_\Sigma (\bar{u}-u)|+r|\nabla_\Sigma^2 (\bar{u}-u)|),
		\end{split}
	\end{equation}
	\begin{equation}\label{eq2.60}
		\begin{split}
			\|E(\bar{u})-E(u)\|_{C^{0,\alpha}_{\beta+2}}\leq C (\|\bar{u}\|_{C^{2,\alpha}_{\beta}}+\|u\|_{C^{2,\alpha}_{\beta}})\|\bar{u}-u\|_{C^{2,\alpha}_{\beta}},
		\end{split}
	\end{equation}
	above $C=C(\Sigma,\alpha,\beta)$.
\end{corollary}

\begin{proof}
	\eqref{eq2.57},\eqref{eq2.58} follow by applying \eqref{eq2.54} to decompose $E(u)$ and \eqref{eq2.55},  \eqref{eq2.56} with $i=j=k=l=0$  to estimate the two terms in the decomposition.
	
	\eqref{eq2.59},\eqref{eq2.60} follow by applying \eqref{eq2.54} to decompose $E(u), E(\bar{u})$ , using the  expanding 
	\begin{equation}
		\begin{split}
			&	E_1(\cdot,\bar{u}(\textbf{x}),\nabla_\Sigma \bar{u}(\textbf{x}),\nabla_\Sigma^2 \bar{u}(\textbf{x}))-E_1(\cdot,u(\textbf{x}),\nabla_\Sigma u(\textbf{x}),\nabla^2_\Sigma u(\textbf{x}))\\
			=&\int_{0}^{1} U(t)\nabla_zE_1(\cdot,U(t),\nabla_\Sigma U(t),\nabla^2_\Sigma U(t))dt\\
			+&\int_{0}^{1} \nabla_\Sigma U(t)\cdot\nabla_{\textbf{q}}E_1(\cdot,U(t),\nabla_\Sigma U(t),\nabla^2_\Sigma U(t))dt\\
			+&\int_{0}^{1} \text{Tr}\bigg(\nabla_\Sigma^2 U(t)\cdot\nabla_{\textbf{A}}E_1(\cdot,U(t),\nabla_\Sigma U(t),\nabla^2_\Sigma U(t))\bigg)dt\\
		\end{split}
	\end{equation}
	\begin{equation}
		\begin{split}
			&	\textbf{E}_2(\cdot,\bar{u}(\textbf{x}),\nabla_\Sigma \bar{u}(\textbf{x}),\nabla_\Sigma^2 \bar{u}(\textbf{x}))-	\textbf{E}_2(\textbf{x},u(\textbf{x}),\nabla_\Sigma u(\textbf{x}),\nabla^2_\Sigma u(\textbf{x}))\\
			=&\int_{0}^{1} U(t)\nabla_z	\textbf{E}_2(\cdot,U(t),\nabla_\Sigma U(t),\nabla^2_\Sigma U(t))dt\\
			+&\int_{0}^{1} \nabla_\Sigma U(t)\cdot\nabla_{\textbf{q}}	\textbf{E}_2(\cdot,U(t),\nabla_\Sigma U(t),\nabla^2_\Sigma U(t))dt\\
			+&\int_{0}^{1} \text{Tr}\bigg(\nabla_\Sigma^2 U(t)\cdot\nabla_{\textbf{A}}	\textbf{E}_2(\cdot,U(t),\nabla_\Sigma U(t),\nabla^2_\Sigma U(t))\bigg)dt\\
		\end{split}
	\end{equation}
	where $U(t)=u+t(\bar{u}-u)$
	and then using \eqref{eq2.55},  \eqref{eq2.56} with $i=0,j+k+l=1$  to estimate the two terms in the expanding.
	
\end{proof}

\subsection{Ancient mean curvature flow}
We are ready to prove the main theorem. We continue to fix $\delta_0\in (0,-\lambda_I), \alpha\in (0,1)$ and $\beta> n$.
\begin{theorem}\label{thm4.1}
	Suppose that $\Sigma^n\subset \mathbb{R}^{n+1}$ is a minimal hypersurface. Let $I$ be the Morse index of $\Sigma$. There exists $\mu_0=\mu_0(\Sigma,\alpha,\delta_0)$ such that for every $\mu\geq \mu_0$, there exists a corresponding $\varepsilon=\varepsilon(\Sigma,\alpha,\delta_0,\mu)$ with the following property:
	For any $\mathbf{a}\in B_{\varepsilon}(\mathbf{0})\subset \mathbb{R}^I$ there exists a unique $\mathscr{S}:\Sigma\times \mathbb{R}_-\to \mathbb{R}$ so that the
	hypersurfaces $S(t ):=\text{ graph }_\Sigma \mathscr{S}(\mathbf{a})(\cdot,t)$ satisfy the mean curvature flow
	\begin{equation}\label{eq4.4}
		\begin{split}
			\frac{\partial }{\partial t}\mathbf{x}=\mathbf{H}_{S(t)}(\mathbf{x}),\; \forall \mathbf{x}\in S(t),
		\end{split}
	\end{equation}
	with the priori decay
	\begin{equation}
		\begin{split}
			\sup\limits_{t\in \mathbb{R}_-}e^{-\delta_0 t}\|(\mathscr{S}(\mathbf{a})-\iota_-(\mathbf{a}))(\cdot,t)\|_{C^{2,\alpha}_\beta}\leq \mu |\mathbf{a}|^2,
		\end{split}
	\end{equation}
	and the terminal condition $\Pi_{<0}(\mathscr{S}(\mathbf{a}))(\cdot,0)=\iota_-(\mathbf{a})(\cdot,0)$.
\end{theorem}

\begin{proof}
	The proof is similar to that in \cite[Theorem 6.1]{CCCS24}. For the readers’ convenience, we give the detailed proof here.	The geometric PDE \eqref{eq4.4} is equivalent to \eqref{eq3.18}. Consider the space 
	\begin{equation}
		\mathscr{C}[\mathbf{a}]:=\{u:\Sigma\times \mathbb{R}_-\to \mathbb{R}:\Pi_{< 0}(u(\cdot,0))=\iota_-(\mathbf{a})(\cdot,0),\|u\|_*<\infty\}
	\end{equation}
	where 
	\begin{equation}
		\|u\|_*:=\sup_{t\in \mathbb{R}_-}e^{-\delta_0 t}\left(\|u(\cdot,t)\|_{C^{2,\alpha}_{\beta}} +\|\frac{\partial}{\partial t}u(\cdot,t)\|_{C^{0,\alpha}_{\beta +2}}     \right).
	\end{equation}
	It is complete with respect to $d_*(u,\bar{u})=\|u-\bar{u}\|_*$ and it's easy to verify that $\|\iota_-(\mathbf{a})\|_*\leq C|\mathbf{a}|.$
	Let $\eta>0$ be as in Lemma \ref{lmm2.11}. For $u\in \mathscr{C}[\textbf{a}], \|u\|_*\leq \eta$, let $\mathscr{S}(u;\mathbf{a})$ be a solution of 
	\begin{equation}\label{eq2.78}
		\left(	\frac{\partial}{\partial t}-L\right)\mathscr{S}(u;\mathbf{a})=E(u)
	\end{equation}
	with $\mathscr{S}(u;\mathbf{a})(\cdot,0)=\iota _-(\mathbf{a})(\cdot,0)$. Equivalently, we are solving
	\begin{equation}
		\left(	\frac{\partial}{\partial t}-L\right)\left(\mathscr{S}(u;\mathbf{a})-\iota _-(\mathbf{a})\right)=E(u),\; \Pi_<0(\mathscr{S}(u;\mathbf{a})-\iota _-(\mathbf{a}))(\cdot,0)=0.
	\end{equation}
	Existence is guaranteed and  the  priori quadratic decay of $u$ implies 
	decay of $E(h)$ by Corollary \ref{coro2.12}. Proposition \ref{prop2.9} and \eqref{eq2.58} 
	imply:
	\begin{equation}
		\sup_{t\leq 0}e^{-\delta_0 t}\|(\mathscr{S}(u;\mathbf{a})-\iota _-(\mathbf{a}))(\cdot,t)\|_{C^{2,\alpha}_{\beta}}\leq C\|u\|_*^2.
	\end{equation}
	So, we have 
	\begin{equation}
		\|\mathscr{S}(u;\mathbf{a})-\iota _-(\mathbf{a})\|_*\leq C\|u\|_*^2.
	\end{equation}
	Therefore $\mathscr{S}(u;\textbf{a})\in \mathscr{C}[\textbf{a}]$. Noting that solutions of \eqref{eq2.78} are uniquely determined within $\mathscr{C}[\textbf{a}]$, $\mathscr{S}(u;\textbf{a})$ is a  well-defined map of small elements
	of $\mathscr{C}[\textbf{a}]$ into $\mathscr{C}[\textbf{a}]$.
	
	Likewise, for $\bar{u}\in \mathscr{C}[\textbf{a}],\|\bar{u}\|_*\leq \eta$, we have 
	\begin{equation}
		\begin{split}
			\left(	\frac{\partial}{\partial t}-L\right)(\mathscr{S}(\bar{u};\mathbf{a})-\mathscr{S}(u;\mathbf{a}))&=E(\bar{u})-E(u)\\
			\Pi_{<0}(\mathscr{S}(\bar{u};\mathbf{a})-\mathscr{S}(u;\mathbf{a}))&=0
		\end{split}
	\end{equation}
	Therefore the discussion above applies with $\bar{u}-u$ in place of $u-\iota_-(\textbf{a})$ and \eqref{eq2.60} gives:
	\begin{equation}
		\|\mathscr{S}(\bar{u};\mathbf{a})-\mathscr{S}(u;\mathbf{a})\|_*\leq
		C(\|\bar{u}\|_*+\|u\|_*)\|\bar{u}-u\|_*.
	\end{equation}
	
	Consider the subset $X:=\{u\in  \mathscr{C}[\mathbf{a}]:\|u-\iota_-(\mathbf{a})\|_*\leq \mu |\mathbf{a}|^2\}$. There exists $\mu_0$ such that, for all $\mu\geq \mu_0$, there exists $\varepsilon$ such that $\mathbf{a}\in B_{\varepsilon}(0)\subset  \mathbb{R}^I$ and $u\in X$ imply $\mathscr{S}(u;\mathbf{a})\in X$. Thus, $\mathscr{S}(\cdot;\mathbf{a})$ maps $X$ into itself. It is a contraction mapping. By the completeness of $X$, there exists a unique fixed point of $\mathscr{S}(\cdot;\mathbf{a}) $ in $X$, which we denote $\mathscr{S}(\mathbf{a})$. The smoothness on $\mathscr{S}(\mathbf{a})$ follows by bootstrapping standard parabolic Schauder estimates.
\end{proof}

\end{document}